\documentclass[preprint,12pt]{elsarticle}

\usepackage{amsmath,amsfonts,calrsfs,amssymb,color}
\usepackage{amsthm}
\usepackage{tikz}
\usetikzlibrary{arrows.meta,positioning}
\makeatletter
\renewenvironment{proof}[1][\proofname]{\par
  \pushQED{\qed}%
  \normalfont \topsep6\p@\@plus6\p@\relax
  \trivlist
  \item[\hskip\labelsep
        \itshape
    #1\@addpunct{}]\ignorespaces
}{%
  \popQED\endtrivlist\@endpefalse
}
\makeatother

\newtheorem{Theorem}{Theorem}[section]
\newtheorem{Definition}{Definition}[section]

\newtheorem{Lemma}[Theorem]{Lemma}
\newtheorem{Corollary}[Theorem]{Corollary}
\newtheorem{Remark}[Theorem]{Remark}

\newcommand{\R}{\mathbb R}
\newcommand{\N}{\mathbb N}

\let\oldproofname=\proofname
\renewcommand{\proofname}{\rm\sc{\oldproofname}}
\newcommand{\eps}{\varepsilon}

\journal{Journal of Functional Analysis}

\begin{document}

\begin{frontmatter}

\title{A Bootstrap Proof of the Abstract Wiener Space Theorem for Fr\'echet Spaces}

\author{Luciano Tubaro}
\address{Dipartimento di Matematica, Universit\`a di Trento, Povo, Italy}

\begin{abstract}
In this paper, we provide an alternative proof of a result due to Gross (1960s), extending it to the Fréchet space setting. Our argument adapts the proof technique developed by Bogachev in the Banach case, as presented in the monograph {\sl Gaussian Measures}, to the more general framework of separable Fréchet spaces, simplifying several steps along the way.
\end{abstract}


\end{frontmatter}

\section{Introduction} In this paper we state and prove Theorem 1 and Theorem 2 for Fréchet spaces, adapting the proof techniques of Theorem 3.9.5 and Theorem 3.9.6, respectively, from the book {\sl Gaussian Measures} by Bogachev, where the analogous results are established in the Banach space setting using essentially the method of Kallianpur \cite{K}. Our goal is to extend this framework while simultaneously simplifying certain steps. Moreover, Definition 1 below generalises Definition 3.9.4 of that reference.

\medskip
For readers not acquainted with the classical theory of abstract Wiener spaces, we recall some known facts. Gross introduced 
 abstract Wiener spaces by considering a real separable Hilbert space $H$ equipped with the standard cylindrical (merely additive) Gaussian measure $\nu$ and a measurable seminorm $q$.
He took the completion $X$ of $H$ with respect to $q$: this is a Banach space, and he showed that the restriction of $\nu$ on cylindrical sets of $X$ extends as a $\sigma$-additive Gaussian measure on the Borel $\sigma$-algebra of $X$, via the  Carath\'eodory extension theorem. Kallianpur later gave an alternative proof using a sequence of Gaussian random variables and, in effect, the Kolmogorov extension theorem.  Both proofs are reported in the monograph by Kuo. Furthermore, Kuo considered the case in which the seminorm takes the form  $q(x)=|Tx|$, with $T$ an injective Hilbert-Schmidt operator;
 in this case $Q$ is measurable, hence the completion $X$ is again a Hilbert space carrying a Gaussian measure. 
 This Gaussian measure has also been studied directly in Kuo and in other references, notably Da Prato–Zabczyk.

 \medskip
The literature on locally convex spaces offers several routes to results of
this type, which differ substantially both in scope and in method.
Dudley, Feldman and Le Cam \cite{DFL} give two results relevant here. Their
Theorem~4 shows that if $\mu$ is a mean-zero Gaussian measure on a general
locally convex space $W$ satisfying a tightness hypothesis --- concentration
on convex, weakly compact sets --- then $\mu = n\circ\theta^{-1}$ for a
continuous linear map $\theta$ from a Hilbert space $H$ into $W$, where $n$
is the canonical normal distribution on $H$. This hypothesis is automatic
for separable Banach spaces, as Dudley--Feldman--Le Cam themselves note, but
is \emph{not} automatic for a general separable Fr\'echet space: $\mu$ is a
priori given only through its finite-dimensional (cylindrical) projections,
on the weak Borel sets $\sigma(W,W')$, and whether such a cylindrical
measure extends at all to a genuinely countably additive, tight Borel
measure for the strong Fr\'echet topology is exactly the abstract-Wiener-space
existence question itself --- their own Corollary~1.1 shows this extension
is automatic only when $W$ is $\sigma$-compact, which no infinite-dimensional
Fr\'echet space is. Supplying this tightness in the Fr\'echet setting is
therefore a substantive task, not a free corollary, and it is exactly what
Rajput, Zapa{\l}a, and the present paper each accomplish by different
routes. Their Theorem~5 treats the Fr\'echet case specifically, but for an
\emph{arbitrary} regular Borel measure $\mu$ on a Fr\'echet space $W$ (not
necessarily Gaussian, and already assumed regular): it
constructs $H$, a cylinder set measure $m$ on $H$ making $x\mapsto(x,\cdot)$
continuous into convergence in probability, and a continuous
$\theta\colon H\to W$ with $m\circ\theta^{-1}=\mu$, via a weighted
$L^2$-embedding technique unrelated to Rajput's or Zapa{\l}a's routes below.
Rajput \cite{R} explicitly notes that his own Fr\'echet-space theorem
``has a point of contact with, and follows from'' their Theorem~4, obtained
however ``by a somewhat different set of ideas''; Rajput's own route is
self-contained and specific to the Gaussian, Fr\'echet-space case: his
Theorem~4.2 shows that every zero-mean Gaussian measure on an arbitrary
separable Fr\'echet space $F$ is the $\sigma$-extension of the canonical
cylindrical measure of a separable Hilbert space $H$, continuously and
densely embedded in $F$, proceeding via the topological support of the
measure, described through the reproducing kernel Hilbert space of its
covariance function, together with an adaptation of the It\^o--Nisio
series-convergence techniques; this requires genuinely new arguments since
the Banach-space techniques of Kallianpur, Jain--Kallianpur, Kuelbs and
Sato --- used for the Banach-space case of the same statement --- do not
carry over verbatim to the Fr\'echet setting. Badrikian and Chevet
\cite{BC} address the more general ``espaces de Wiener'' problem of when a
weakly continuous linear map from a Hilbert space into an arbitrary
topological vector space (not necessarily locally convex) carries the
canonical Gaussian cylinder measure to a Radon measure, of which the
Fr\'echet case is only a special instance. Borell \cite{BO}, building on his
own theory of convex measures on locally convex spaces, develops the general
theory of Radon Gaussian measures on arbitrary locally convex spaces, without
singling out the Fr\'echet case. Zapa{\l}a \cite{Z} generalises Gross's
original completion method directly to an arbitrary separable Fr\'echet
space $M$ --- completing $H$ with respect to a measurable seminorm exactly as
Gross does for Banach spaces --- and, along the way, establishes a
Fr\'echet-space version of the Brunn--Minkowski inequality for Gaussian
measures.

\medskip
Our approach differs from all of these not merely in the class of spaces
considered, but in method. Rather than working at the level of general
locally convex spaces via a tightness/weak-compactness condition, as
Dudley--Feldman--Le Cam's Theorem~4 and Borell do, embedding through a
weighted $L^2$ construction for arbitrary regular measures, as
Dudley--Feldman--Le Cam's Theorem~5 does, adapting the completion argument
directly to $F$, as Zapa{\l}a does, or working through the topological
support and its RKHS description, as Rajput does, we bootstrap through an
auxiliary Hilbert-space abstract Wiener space $(E,H,\nu_1)$ together with an
explicit Borel linear isomorphism $L\colon E_0\to F_0$ between conull sets (Corollary~3
below),
transporting the classical Hilbert/Banach-space theory of Gross and
Kallianpur verbatim to the Fr\'echet setting.

\medskip
A close antecedent, in the classical Banach-space abstract-Wiener setting, is Fernholz
\cite{Fe}, who shows that the measure on any infinite-dimensional abstract Wiener space can be
transformed into that of any other by a measurable linear transformation; this is not automatic
from our results, nor conversely, and the two constructions should not be confused. Fernholz
works throughout with Banach codomains, and the transformation is not required to intertwine the
Cameron--Martin embeddings, nor to arise from an explicit series construction admitting a Borel
inverse on a specified conull set. Here the codomain is Fr\'echet, the map $L$ of
Corollary~\ref{cor.transport} is exhibited explicitly rather than obtained abstractly, restricts
to the identity of $H$ under the two Cameron--Martin embeddings ($L\circ j_E=\iota$), and is a
genuine Borel isomorphism $E_0\to F_0$ between specified conull sets -- exactly the compatibility
with the Cameron--Martin structure that the nonlinear-transformation applications of Section~4
require. Accordingly the contribution of this paper is better stated, soberly, as: (1) a new
proof of Gross's extension to Fr\'echet spaces (Theorem~\ref{uno}--\ref{due}), simplifying and, in
two places, correcting the argument as presented by Bogachev; (2) the explicit construction of a
linear Borel isomorphism between conull sets, compatible with the Cameron--Martin embedding
(Corollary~\ref{cor.transport}), in the spirit of, but not subsumed by, Fernholz's theorem; and
(3) the use of this isomorphism to transport Bogachev's local nonlinear change-of-variables
theorem (Section~4) to the Fr\'echet setting. We do not claim that Gross's extension to
Fr\'echet spaces is itself new, since it already follows, by different routes, from Rajput
\cite{R} and Zapa{\l}a \cite{Z}.

\section{Some Prerequisites}
The following  Definition and Lemma are taken verbatim from Bogachev's {\sl Gaussian Measures} 
(Definition 3.9.2 and Lemma 3.9.3).

\medskip
{\bf Definition (Gross measurable seminorm)}.{\it  Let $H$ be a separable Hilbert space and let $P(H)$ be the set of all orthogonal projections in $H$ with finite-dimensional range. A seminorm $q$ on $H$ is called measurable in the sense of Gross (or Gross measurable) if, for every $\eps>0$, there exists a finite-dimensional orthogonal projection $P_\eps\in P(H)$ such that
$$\nu(x\in H\colon q(Px)>\eps)<\eps,\qquad\forall P\in P(H), P\perp P_\eps,$$
 where $\nu$ is the canonical cylindrical Gaussian measure on $H$.}\\
 The set appearing on the left-hand side is $\nu$-measurable, since the function $q$ is continuous on every finite-dimensional subspace. The following lemma shows that $q$ is in fact continuous on all of $H$.

{\bf Lemma}. {\it Every seminorm measurable in the sense of Gross is continuous.}

 \begin{Definition}[\bf Abstract Wiener space over a Fréchet space]\label{d.one}
 A triple $(\iota,H,F)$ is called an abstract Wiener space if $F$ is a separable Fr\'echet space,
 $H$ is a separable Hilbert space, $\iota\colon H\to F$ a continuous linear embedding with dense range,
 and the defining seminorms $q_k$ of $F$, or more precisely the compositions $q_k\circ \iota$, are Gross measurable on $H$. We assume as usual, $q_1\le q_2\le\ldots$.
 \end{Definition}
 
 We also introduce  the associated quasinorm (also called pseudonorm)
 $$
|x| = \sum_{k=1}^\infty\frac{1}{2^k}\frac{q_k(x)}{1+q_k(x)}
$$
The quasinorm $|\cdot|$ is Gross measurable. Indeed, setting
$$
r_n(x)=\sum_{k=n+1}^\infty\frac{1}{2^k}\frac{q_k(x)}{1+q_k(x)}
$$
one has $r_n(x)\to 0$ as $n\to\infty$,uniformly. To verify the Gross measurability of $|\cdot|$, let $\eps>0$
and write
$$
\nu\big\{x\in H :\;\;|Px|>\varepsilon\big\}=\nu\Big\{x\in H :\;\; \sum_{k=1}^n\frac{1}{2^k}\frac{q_k(Px)}{1+q_k(Px)}+r_n(Px)>\varepsilon\Big\}
$$
Since
$$
\sum_{k=1}^n\frac{1}{2^k}\frac{q_k(Px)}{1+q_k(Px)}+r_n(Px)\le \sum_{k=1}^n\frac{1}{2^k}\;q_k(Px)+\frac{1}{2^n}
\le q_n(Px)+\frac{1}{2^n},
$$
given $\eta>0$ with $0<\eta<\varepsilon$, we choose $n$ such that $\tfrac{1}{2^n}<\eta$, set
$\lambda:=\varepsilon-\eta$, and choose a Gross witness $P_n$ for the seminorm $q_n$ at threshold
$\lambda$, i.e.\ $P_n$ such that $\nu\big\{x\in H:\;q_n(Px)>\lambda\}<\lambda$ for all
$P\perp P_n$. Then $\{|Px|>\varepsilon\}\subseteq\{q_n(Px)+\eta>\varepsilon\}=
\{q_n(Px)>\varepsilon-\eta\}=\{q_n(Px)>\lambda\}$, so
$$
\nu\big\{x\in H :\;\;|Px|>\varepsilon\big\}\le \nu\big\{x\in H :\;\;q_n(Px)>\lambda\}<\lambda=\varepsilon-\eta< \varepsilon
$$
 \section{Main results}
 \begin{Theorem}\label{uno}
 Let $(\iota,H,F)$ be an abstract Wiener space as in the Definition 1. Then the canonical cylindrical Gaussian measure $\nu$ on $H$
 extends to a countably additive Gaussian measure on $F$ (via $\iota$, i.e.\ $\nu\circ\iota^{-1}$ is countably additive). Moreover, $H$ coincides with the Cameron--Martin space
 of this measure on $F$.
  \end{Theorem}
  \begin{Theorem}\label{due}
Let $\gamma$ be a centered Gaussian measure on a separable Fr\'echet space $F$ such that the Cameron-Martin space
$H=H(\gamma)$
is dense in $F$, and let $\iota\colon H\to F$ be the natural embedding. Then $(\iota,H,F)$ is an abstract Wiener space.
\end{Theorem}

 \begin{proof} {\sc of Theorem \ref{uno}}

The finite-dimensional case ($\dim H<\infty$) is immediate: $F=\iota(H)$ up to the topology, and
$\nu$ is already countably additive. Henceforth assume $H$ is infinite-dimensional (so that,
e.g., $r_n\uparrow\infty$ in Step~2 below is meaningful).
 
 \noindent {Step 1.} Denote by $|\cdot|_F$ the quasinorm on $F$, Gross measurable by the computation
above. For each $j\in\N$, apply Gross measurability with $\varepsilon=2^{-j}$ to obtain a
finite-dimensional orthogonal projection $Q_j\in P(H)$ such that
$\nu(x: |Px|_F>2^{-j})<2^{-j}$ for every $P\in P(H)$ with $P\perp Q_j$. Set
$P_n:=Q_1\vee\cdots\vee Q_n$ (the orthogonal projection onto $\mathrm{Ran}(Q_1)+\cdots
+\mathrm{Ran}(Q_n)$): an increasing sequence of finite-dimensional projections with $P_n\to I$
and $P_n\ge Q_n$. Set $c_n:=2^{-n}$. For $m\ge n$, $P_m-P_n$ is an orthogonal projection with
$\mathrm{Ran}(P_m-P_n)\perp\mathrm{Ran}(P_n)\supseteq\mathrm{Ran}(Q_n)$, i.e.\ $P_m-P_n\perp Q_n$;
applying Gross measurability with $P=P_m-P_n$, $\varepsilon=c_n$ gives
\begin{equation}
\label{eq.incr}
 \nu\Big(x\colon |(P_m-P_n)x|_F>c_n\Big)<c_n,\quad\forall m\ge n.
\end{equation}

\begin{Remark}[Correcting Bogachev's relation (3.9.2)]\label{rem.pmpn}
As printed in Bogachev's proof of Theorem~3.9.5, and in an earlier draft of this Step matching
it verbatim, the analogous relation -- there labelled $(3.9.2)$ -- reads
$\nu(x: |P_mx|_F>c_n)<c_n$, with $P_m$ in place of $P_m-P_n$. This cannot be correct: since
$P_n\to I$ strongly and $|\cdot|_F$ is continuous (Lemma above), $P_mx\to x$ and hence
$|P_mx|_F\to|x|_F$ as $m\to\infty$, which is not small for typical $x$. What Step~3 below
actually needs is control of the increment $S_m-S_n$, governed by $|(P_m-P_n)x|_F$, not by
$|P_mx|_F$ alone; \eqref{eq.incr} is exactly what Gross measurability delivers once applied,
correctly, with $P=P_m-P_n\perp Q_n$ as checked above. (The freedom to enlarge a Gross witness,
used to define $P_n$, is immediate from the definition: if $P_\varepsilon$ works, so does any
finite-rank $P_\varepsilon'\ge P_\varepsilon$, since
$\{P:P\perp P_\varepsilon'\}\subseteq\{P:P\perp P_\varepsilon\}$.)
\end{Remark}

{Step 2.} \emph{$\bigcup_nP_n(H)$ is dense in $H$.} If $h\ne0$ and $h\perp P_n(H)$ for every
$n$, then in particular $h\perp Q_j(H)$ for every $j$ (since $Q_j(H)\subset P_n(H)$ for $n\ge
j$), so the rank-one projection $P_h$ onto $\R h$ satisfies $P_h\perp Q_j$ for every $j$; since
$Q_j$ witnesses Gross measurability of the quasinorm $|\cdot|_F$ itself at threshold $2^{-j}$
(not, a priori, of any single seminorm $q_k$ -- this is the witness the construction actually
supplies, and is what must be used here), this gives directly
$$
\nu\bigl(x: |P_hx|_F>2^{-j}\bigr) < 2^{-j}, \qquad \forall j\in\N.
$$
Write $P_hx = Z\,\widehat h$ with $Z:=(x,h)_H/|h|_H\sim\mathcal N(0,1)$ under $\nu$ and
$\widehat h:=h/|h|_H$ a fixed unit vector; set $c:=|\iota(\widehat h)|_F>0$ (positive since
$\iota$ is injective and $\widehat h\ne0$). Each term $q_k(t\widehat h)/(1+q_k(t\widehat h))$ of
the quasinorm is non-decreasing in $|t|$ (as $q_k(t\widehat h)=|t|\,q_k(\widehat h)$ by
homogeneity of the seminorm $q_k$, and $s\mapsto s/(1+s)$ is increasing), so $t\mapsto
|t\widehat h|_F$ is non-decreasing in $|t|$; in particular $|t\widehat h|_F\ge
|\widehat h|_F=c$ for every $|t|\ge1$. Hence, for every $j$,
$$
\nu\bigl(x: |P_hx|_F>2^{-j}\bigr) \ge \nu\bigl(|Z|\ge1,\ |Z\widehat h|_F\ge c\bigr)
\ge \Pr(|Z|\ge1) \quad \text{once } 2^{-j}<c,
$$
a fixed positive constant independent of $j$. This contradicts
$\nu(|P_hx|_F>2^{-j})<2^{-j}\to0$ for $j$ large enough that $2^{-j}<\min(c,\Pr(|Z|\ge1))$. Hence
no such $h\ne0$ exists: $\bigcup_nP_n(H)$ is dense in $H$.

\emph{Choice of basis.} Set $r_0:=0$, $r_n:=\dim P_n(H)$ (finite, and $r_n\uparrow\infty$ since
$\bigcup_nP_n(H)$ is dense and infinite-dimensional, $H$ being infinite-dimensional). Choose,
inductively, an orthonormal basis $e_{r_{n-1}+1},\ldots,e_{r_n}$ of $P_n(H)\ominus P_{n-1}(H)$ for
each $n\ge1$; this produces a complete orthonormal basis $(e_k)_{k\ge1}$ of $H$ with
$P_n(H)=\mathrm{span}\{e_1,\ldots,e_{r_n}\}$ exactly, i.e.\ $P_n=\widetilde P_{r_n}$ where
$\widetilde P_m$ denotes the orthogonal projection onto $\mathrm{span}\{e_1,\ldots,e_m\}$. With
this basis, \eqref{eq.incr} reads $\nu(x:|(\widetilde P_{r_m}-\widetilde P_{r_n})x|_F>c_n)<c_n$
for $m\ge n$, giving directly the Cauchy property, in $\nu_1$-measure, of the subsequence
$(S_{r_n})_n$ (Step~3 below). To pass from this subsequence to the full sequence $(S_k)_k$: for
$r_n<k\le r_{n+1}$, $\widetilde P_k-\widetilde P_{r_n}$ is the orthogonal projection onto
$\mathrm{span}\{e_{r_n+1},\ldots,e_k\}\subset P_{n+1}(H)\ominus P_n(H)$, hence
$\mathrm{Ran}(\widetilde P_k-\widetilde P_{r_n})\perp P_n(H)\supseteq Q_n(H)$, and Gross
measurability (witness $Q_n$, threshold $c_n$) gives
$\nu(x:|(\widetilde P_k-\widetilde P_{r_n})x|_F>c_n)<c_n$ for every such $k$; combined with the
subsequential estimate this controls every increment $S_k-S_{r_n}$, $r_n\le k\le r_{n+1}$,
uniformly in $k$, and hence the full sequence $(S_k)_k$ is Cauchy in $\nu_1$-measure.
Consider the Hilbert space $E$, obtained as the completion of $H$ with respect to the
norm
$$
x \mapsto \Big\{ \sum_{n=1}^\infty \frac{(x,e_n)_H^2}{n^{2}}\Big\}^{1/2}.
$$
Here $E$ is understood as the completion, with respect to this norm, of the (dense) subspace
$\tilde H\subset H$ of finite linear combinations of the $e_n$'s.
Let $j_E$ denote the natural embedding $H\to E$. Applying Theorem 2.3.1 (of Bogachev \cite{BG}), we conclude that
the pushforward $\nu_1 = \nu\circ j_E^{- 1}$ is countably additive gaussian measure on the auxiliary Hilbert space $E$,
and one has $H = H(\nu_1)$.\\[6pt]
 \noindent {Step 3.} Consider the sequence of $F$-valued random vectors on $(E, \nu_1)$ defined by
$$
S_n(x) = \sum_{k=1}^n k^2\, (x,e_k)_E\,\iota(e_k)
$$
(the exponent $2$, rather than $1$, is required for consistency with Step~4 below; see Remark~\ref{rem.exponent}.
We also make explicit here, as $\iota(e_k)\in F$, the embedding of the basis vectors $e_k\in H$ into $F$.)
%\red
This is precisely the point where Kallianpur's approach \cite{K} enters: rather than extending the measure
$\nu$ via Carath\'eodory's theorem as in Gross's original argument, one realises the limiting object directly
as an a.s.\ limit of the partial sums $S_n$, in the spirit of the Kolmogorov extension theorem for the sequence
of Gaussian random variables $(\xi_k)_k$.
%\nor
Relationship \eqref{eq.incr} implies that 
$\{S_n\}$ is a Cauchy sequence  in $\nu_1$-measure. Since $F$ is a complete metric space, the sequence converges in measure to a random vector $S\colon E\to F$ (the argument is the same as for real-valued random variables; see Problem 3.11.38 in Bogachev \cite{BG}).
%\red
More precisely, fix a countable family of seminorms $(q_k)_k$ generating the topology of $F$, and let
$d(y,z)=\sum_k 2^{-k} \frac{q_k(y-z)}{1+q_k(y-z)}$ be the associated complete translation-invariant metric on $F$.
Cauchy in $\nu_1$-measure with respect to $d$ means that for every $\eta>0$,
$\nu_1(x: d(S_n(x),S_m(x))>\eta)\to 0$ as $n,m\to\infty$. As in the real-valued case, one extracts a subsequence
$(S_{n_j})$ that is $\nu_1$-almost surely Cauchy in the complete metric space $(F,d)$; completeness of $F$ then
yields an $F$-valued limit $S(x) = \lim_j S_{n_j}(x)$ for $\nu_1$-almost every $x$. In particular $S$ is a genuine
$F$-valued (Borel) random vector, not merely a formal limit. Since every subsequence of $(S_n)$ is again
Cauchy in $\nu_1$-measure, the same argument shows that every subsequence of $(S_n)$ has in turn a further
subsequence converging $\nu_1$-a.s.\ to $S$; this is enough to conclude that the full sequence $(S_n)$
converges to $S$ in $\nu_1$-measure (a sequence converging in measure to two different limits along
subsequences cannot occur, since a.s.\ convergent subsequences to different limits would contradict the
Cauchy property in measure of the original sequence).\\[6pt]
%\nor
 \noindent {Step 4.} It remains to verify that the distribution $\nu_2:= S^{-1}\nu_1$ of $S$
 coincides with the cylindrical measure $\nu\circ \iota^{-1}$ on $F$. Let $f \in F^*$; by
 the Riesz representation theorem, there exists a unique $v \in H$ such that $f(h) =(v,h)_H$ for all $h\in H$.
 The distribution $\nu_2\circ f^{-1}$ is the law of $f(S)$ under $\nu_1$. Since $S_n\to S$ in $\nu_1$-measure,
 we have 
$$
f(S) = \sum_{k=1}^\infty k^2\,(x,e_k)_E\, f(\iota(e_k)) = \sum_{k=1}^\infty k^2\,(x,e_k)_E\, (v,e_k)_H
$$
The functionals $\xi_k(x) = k^2\, (x,e_k)_E$ are independent standard Gaussian
random variables on $(E, \nu_1)$, since their restrictions to $H = H(\nu_1)$ are
given by the vectors $e_k$ of an orthonormal basis: indeed, from the norm defining
$E$ in Step~2, $(x,e_k)_E = k^{-2}(x,e_k)_H$ for $x\in H$, so that
$k^2(x,e_k)_E = (x,e_k)_H$ on $H$.
Hence, the variance of $f (S)$ equals
$|v|_H^2 = \sum_{k=1}^\infty(v,e_k)_H^2$.

On the other hand,  $\nu\circ \iota^{-1} \circ f^{-1}$ is a centred
Gaussian measure with the same variance $\sigma^2 = |v|_H^2$, since $f\circ \iota$ 
is given on $H$ by the vector $v$. The two distributions therefore coincide. An identical argument
applied to any $g \in F^*$ shows that $(g, f )_{L^2(\nu_2)}=(u, v)_H$ where
$u \in H$ corresponds to $g$ which yields that $H = H(\nu_2)$.

This completes the proof of Theorem 1. 
\end{proof}

\medskip
The proof above does more than establish the countable additivity of $\nu$ on $F$:
it produces, as a by-product of Steps~3--4, an explicit measurable linear
identification between the auxiliary Hilbert space $E$ and $F$. We record this
separately, since it is the fact used systematically in the companion paper on
nonlinear transformations of Gaussian measures on Fr\'echet spaces to transport
regularity classes and the Fredholm--Carleman density from $E$ to $F$.

\begin{Lemma}[Density of $R(F^*)$ in $H$]\label{lem.dense}
For $f\in F^*$ let $v_f\in H$ be the unique vector with $f(\iota h)=(v_f,h)_H$ for all $h\in H$
(Riesz representation of $f\circ\iota\in H^*$; unique because $\iota$ has dense range), and set
$R(F^*):=\{v_f: f\in F^*\}\subset H$. Then $R(F^*)$ is dense in $H$.
\end{Lemma}
\begin{proof}
If $w\in H$ is orthogonal to $R(F^*)$, then $f(\iota w)=(v_f,w)_H=0$ for every $f\in F^*$. Since
$F$ is locally convex, $F^*$ separates points of $F$ (Hahn--Banach), so $\iota w=0$; as $\iota$
is injective, $w=0$.
\end{proof}

\begin{Corollary}[Transport map]\label{cor.transport}
Write $\gamma:=\nu\circ\iota^{-1}$, the measure on $F$ furnished by Theorem~\ref{uno}. There
exist a separable Hilbert space $E$, a countably additive Gaussian measure $\nu_1$ on $E$ with
$H(\nu_1)=H$, a continuous linear embedding $j_E\colon H\to E$ with dense range, and a
$\nu_1$-a.e.\ defined, measurable, linear map $L\colon E\to F$, injective at every point of its
domain, such that
$$
\gamma = \nu_1\circ L^{-1}, \qquad L\circ j_E = \iota \ \ \text{on } H.
$$
Equivalently, $L$ is a measure isomorphism between $(E,\nu_1)$ and $(F,\gamma)$ that restricts,
under the two embeddings $j_E$ and $\iota$, to the identity map of $H$.
\end{Corollary}

\begin{proof}
\emph{Choice of basis.} By Lemma~\ref{lem.dense}, $R(F^*)$ is dense in $H$; applying
Gram--Schmidt to a countable dense subset of $R(F^*)$ (Gram--Schmidt forms only finite linear
combinations, which stay inside the linear space $R(F^*)$) produces a complete orthonormal basis
$(e_n)_{n\ge1}$ of $H$ with $e_n=v_{f_n}$ for suitable $f_n\in F^*$. Since $(f,g)_{L^2(\gamma)}=
(v_f,v_g)_H$ for $f,g\in F^*$ (the covariance-isometry identity underlying $\gamma=\nu\circ
\iota^{-1}$, and used again in the proof of Theorem~\ref{due} below), $(f_n)$ is orthonormal in
$L^2(\gamma)$, and complete there because $(e_n)$ is complete in $H=H(\gamma)$.

\emph{Construction of $E,\nu_1,j_E,L$.} Build $E$, $\nu_1$, $j_E$ exactly as in Step~2 of the
proof of Theorem~\ref{uno}, but with this basis $(e_n)$ in place of the one produced there (the
two constructions need not coincide): $E$ is the completion of the finite linear span of the
$e_n$ with respect to $x\mapsto\bigl\{\sum_n(x,e_n)_H^2/n^{2}\bigr\}^{1/2}$, and $\nu_1=\nu\circ j_E^{-1}$
is, as in Step~2, countably additive on $E$ with $H(\nu_1)=H$. Set $\eta_k(x):=k^2(x,e_k)_E$ and
$$
S_n(x):=\sum_{k=1}^n\eta_k(x)\,\iota(e_k), \qquad x\in E.
$$

\emph{Convergence.} $\gamma$ is now known (Theorem~\ref{uno}) to be a countably additive Radon
Gaussian measure on $F$ with $H(\gamma)=H$, and $(f_n)$ is a complete orthonormal basis of
$F^*_\gamma$; Bogachev's Theorem~3.5.1 therefore gives $\widetilde P_nx\to x$ in $F$ for
$\gamma$-a.e.\ $x\in F$, where $\widetilde P_nx:=\sum_{i=1}^nf_i(x)\iota(e_i)$. Let
$$
C:=\Big\{(a_k)_{k\ge1}\in\R^{\N}: \textstyle\sum_{k=1}^na_k\,\iota(e_k) \text{ converges in } F
\text{ as } n\to\infty\Big\},
$$
a Borel subset of $\R^{\N}$ (with the product topology), determined solely by the fixed vectors
$\iota(e_k)\in F$ and not by any particular measure. Under $\gamma$, the sequence
$(f_k(x))_{k\ge1}$ has law $\rho:=$ the standard Gaussian product measure on $\R^\N$, since
$(f_k)$ is a complete orthonormal basis of $F^*_\gamma$ (jointly Gaussian, pairwise uncorrelated
hence independent, each $N(0,1)$ since $\|f_k\|_{L^2(\gamma)}=1$); Theorem~3.5.1 says precisely
that $(f_k(x))_k\in C$ for $\gamma$-a.e.\ $x$, i.e.\ $\rho(C)=1$. Under $\nu_1$, the sequence
$(\eta_k(z))_{k\ge1}$ \emph{also} has law $\rho$ -- this is exactly the computation in the
paragraph below, showing the $\eta_k$ are i.i.d.\ standard Gaussian on $(E,\nu_1)$, and does not
presuppose convergence of $(S_n)$. Hence
$$
\nu_1\big(z\in E:\ (\eta_k(z))_k\in C\big) = \rho(C) = 1,
$$
i.e.\ $S_n(z)=\sum_{k=1}^n\eta_k(z)\iota(e_k)$ converges in $F$ for $\nu_1$-a.e.\ $z\in E$. (No
statement about $\gamma$-a.e.\ convergence is ``restricted to $H$'': $H$ has $\gamma$-measure
zero in infinite dimensions, exactly like $j_E(H)\subset E$ has $\nu_1$-measure zero, and $\nu$
is only a finitely additive cylindrical measure on $H$ itself; the argument above transfers the
convergence statement between $F$ and $E$ solely through the \emph{common law} $\rho$ of the two
coordinate sequences on $\R^\N$, never through $H$.) Let $E_0\subset E$ (of full $\nu_1$-measure)
be this convergence set; $E_0=\{z\in E: (S_n(z))_n \text{ converges}\}$ is, term by term, closed
under sums and scalar multiples of convergent sequences, hence a genuine linear subspace of $E$.
Moreover $j_E(H)\subset E_0$: for $h\in H$, $S_n(j_E(h))\to\iota(h)$ for \emph{every} (not merely
$\nu_1$-a.e.) $h$, by the identity $k^2(h,e_k)_E=(h,e_k)_H$ established in Step~4 of the proof of
Theorem~\ref{uno} and continuity of $\iota$ (this is proved again, independently, below). Since
$E_0$ is linear and contains $j_E(H)$, $E_0+j_E(H)=E_0$: $E_0$ is stable under Cameron--Martin
translations. Set $L(x):=\lim_nS_n(x)$ for $x\in E_0$; $L$ is linear on $E_0$, each $S_n$ being
linear.

\emph{$\nu_1\circ L^{-1}=\gamma$.} Identical to Step~4 of the proof of Theorem~\ref{uno}: for
$f\in F^*$ with Riesz representative $v\in H$, $f(L(x))=\sum_k\eta_k(x)(v,e_k)_H$, the
functionals $\eta_k$ are i.i.d.\ standard Gaussian on $(E,\nu_1)$, and the resulting variance
computation shows $f_*(\nu_1\circ L^{-1})=f_*\gamma$ for every $f\in F^*$. Since $F^*$ separates
points of the separable Fr\'echet space $F$ and generates its Borel $\sigma$-algebra,
$\nu_1\circ L^{-1}=\gamma$.

\emph{Injectivity.} Let $N:=\{z\in E: S_n(z)\to0 \text{ in } F\}$, a linear subspace of $E$
(kernel of the sequence of linear maps $S_n$: if $S_n(z)\to0$ and $S_n(z')\to0$ then
$S_n(z+z')=S_n(z)+S_n(z')\to0$, and likewise for scalar multiples). Fix $z\in N$ and $j\in\N$.
Since $v_{f_j}=e_j$,
$$
f_j(S_n(z)) = \sum_{k=1}^n\eta_k(z)\,(e_j,e_k)_H = \eta_j(z), \qquad n\ge j,
$$
a finite sum, constant once $n\ge j$: this is simply the continuous functional $f_j$ applied to
the terms of the (literally, pointwise) convergent sequence $S_n(z)\to0$, with no
measure-theoretic content. Hence $\eta_j(z)=\lim_nf_j(S_n(z))=f_j(0)=0$. As $j$ was arbitrary,
$\eta_j(z)=0$ for every $j$, i.e.\ $(z,ke_k)_E=\eta_k(z)/k=0$ for every $k$; since $(ke_k)_k$ is
a complete orthonormal basis of $E$ (Step~2 of Theorem~\ref{uno}'s proof, applied to this
basis), $z=0$. Thus $N=\{0\}$.

If $x,y\in E_0$ and $L(x)=L(y)$, set $z:=x-y$; then $S_n(z)=S_n(x)-S_n(y)\to L(x)-L(y)=0$ in
$F$, so $z\in N=\{0\}$ and $x=y$. Hence $L$ is injective at every point of $E_0$.

\emph{$L\circ j_E=\iota$ on $H$.} For $h\in H$,
$$
S_n(j_E(h)) = \sum_{k=1}^n k^2(h,e_k)_E\,\iota(e_k) = \sum_{k=1}^n (h,e_k)_H\,\iota(e_k)
\xrightarrow[n\to\infty]{} \iota(h) \quad\text{in } F,
$$
using $k^2(h,e_k)_E=(h,e_k)_H$, continuity of $\iota$, and $\sum_{k=1}^n(h,e_k)_He_k\to h$ in
$H$; hence $L(j_E(h))=\iota(h)$ for every $h\in H$.

\emph{Conclusion.} $E_0$ is Borel in the Polish space $E$ (a countable intersection of Borel
conditions on the continuous maps $S_n$, expressed via the complete translation-invariant metric
of $F$), and $L\colon E_0\to F$ is Borel measurable (a pointwise limit of continuous maps) and
injective. By the Lusin--Souslin theorem \cite{Ke}, $L(E_0)$ is Borel in $F$ and
$L\colon E_0\to L(E_0)$ is a Borel isomorphism; since $\nu_1\circ L^{-1}=\gamma$ and
$\nu_1(E_0)=1$, also $\gamma(L(E_0))=1$. This is precisely the asserted measure isomorphism.
\end{proof}

\begin{Remark}\label{rem.exponent}
The exponent $2$ in Steps~3--4 corrects the corresponding step in the proof of
Theorem 3.9.5 in Bogachev \cite{BG} (pp.~137--138), where the analogous
functionals are written with a single power of $k$. With that convention one
finds, from the norm defining $E$, that $k(x,e_k)_E = k^{-1}(x,e_k)_H$ on $H$,
so that the resulting $\xi_k$ would have variance $k^{-2}$ rather than $1$,
and the variance computed for $f(S)$ at the end of Step~4 would be
$\sum_k k^{-2}(v,e_k)_H^2$ instead of $|v|_H^2$ --- inconsistent with the stated
conclusion $H=H(\nu_2)$. This is consistent with the classical construction of
Kallianpur \cite{K} (as reported in Kuo \cite{Ku}), where the analogous
functionals $n(e_j)$ enter the partial sums with no rescaling at all, matching
$\xi_k=k^2(x,e_k)_E=(x,e_k)_H$ above rather than $k(x,e_k)_E$.
\end{Remark}

\begin{Remark}[The transport diagram]\label{rem.diagram}
The construction is summarised by the diagram
\begin{center}
\begin{tikzpicture}[node distance=2.3cm, auto, every node/.style={font=\small}]
\node (H) {$H$};
\node (F) [right=of H] {$F$ \ (Fr\'echet, $\nu\circ\iota^{-1}$)};
\node (E) [below=of H] {$E$ \ (Hilbert, $\nu_1$)};
\draw[-{Latex}] (H) -- node {$\iota$} (F);
\draw[-{Latex}] (H) -- node[left] {$j_E$} (E);
\draw[-{Latex}, dashed] (E) -- node[below] {$L$} (F);
\end{tikzpicture}
\end{center}
Here $\iota$ and $j_E$ are the two embeddings of $H$ furnished by
Definition~\ref{d.one} and by the construction of $E$ in the proof of
Corollary~\ref{cor.transport}; $L$ is the injective, measurable linear map of that Corollary
closing the triangle, $L\circ j_E=\iota$. Unlike the map $S$ produced along the way in
Theorem~\ref{uno} (which need not be injective, since its basis need not lie in $R(F^*)$),
producing an $L$ with the additional, genuine (pointwise, not merely mod null sets) injectivity
needed here requires the separate argument of Corollary~\ref{cor.transport}: choosing the basis
inside $R(F^*)$, and invoking Bogachev's Theorem~3.5.1 -- available only once $\gamma$ is
already known to be countably additive, hence only \emph{after}, and not during, the proof of
Theorem~\ref{uno} -- in place of the estimate \eqref{eq.incr}.
\end{Remark}

\begin{Remark}[Domain of $L$]\label{rem.domain}
$L$ is defined only on the $\nu_1$-full-measure set $E_0\subset E$ where the sequence $(S_n)$ of
Corollary~\ref{cor.transport} converges; we extend it arbitrarily (say, by $0$) outside $E_0$.
On $E_0$ itself, however, $L$ is injective at \emph{every} point, not merely up to a further
null set (Corollary~\ref{cor.transport}, \emph{Injectivity}); this is what the Lusin--Souslin
argument there upgrades to a genuine Borel isomorphism $E_0\to L(E_0)$, with $L(E_0)$ of full
$\gamma$-measure. All statements below involving $L^{-1}$, and the maps
$T_n=L\circ\widehat T_n\circ L^{-1}$ of Section~4 in particular, are understood as defined on
$E_0$, respectively $L(E_0)$, and the resulting bijections and measure identities hold
everywhere on these full-measure sets -- not merely up to further null-set adjustments.
\end{Remark}

\begin{proof}{\sc of Theorem \ref{due}}
Let $\{\xi_n\}$ be an orthonormal basis $F^*_\gamma := \overline{F^*}^{\,L^2(\gamma)}$, the closure of $F^*$ in $L^2(\gamma)$,
with $\xi_n \in F^*$. Set $e_i = R_\gamma(\xi_i)$ (the image of $\xi_i$ under the
covariance operator) and define
\[
  P_n x = \sum_{i=1}^n \xi_i(x)\, e_i.
\]
Since $\{\xi_i\}$ is orthonormal in $L^2(\gamma)$ and $R_\gamma$ is, by
definition of the covariance operator, an isometry from $F^*_\gamma$ onto $H$,
the vectors $e_i=R_\gamma(\xi_i)$ are orthonormal in $H$; hence $P_n$ is the
orthogonal projection of $H$ onto $\mathrm{span}\{e_1,\dots,e_n\}$.
Since the variance of every $f \in F^*$ with respect to $\gamma$ equals $\|f|_H\|_H^2$,
the measure $\gamma$ coincides with the image of the standard cylindrical Gaussian measure
$\nu$ on $H$ under the natural embedding $\iota\colon H \to F$:
\[
  \gamma = \nu \circ \iota^{-1}.
\]

\medskip
\noindent\textbf{Step 1: Gross measurability of a seminorm.}
We verify the Gross measurability of an arbitrary defining seminorm $q$ of $F$.
For any $\varepsilon > 0$, we claim that
\[
  \lim_{n\to\infty} \gamma\bigl(x \in F : q(x - P_n x) > \varepsilon\bigr) = 0.
\]
Indeed, by Bogachev's Theorem~3.5.1, $P_n x \to x$ in $F$ for $\gamma$-almost every $x$.

\begin{Remark}\label{rem.itonisio}
Bogachev's Theorem~3.5.1 is stated in \cite{BG} for Banach spaces, but it is in
substance a special case of the It\^o--Nisio theorem for independent, symmetric,
random elements: $(P_nx)_{n\ge1}$ is the sequence of partial sums of the independent
symmetric $H$-valued (hence $F$-valued, via $\iota$) summands $(P_n-P_{n-1})x$, and
a.s.\ convergence is equivalent, by that theorem, to convergence in law of the
partial sums together with tightness of the limit -- both automatic here since
$\gamma$ is already known to be a Radon Gaussian measure on $F$. The classical proof
of It\^o--Nisio (via the L\'evy--Ottaviani symmetrization inequality and Prokhorov's
tightness criterion) uses only the translation-invariant metric of $F$ and the fact
that $F$ is Polish, never the homogeneity of a Banach norm; it therefore applies
verbatim to any separable Fr\'echet space. This extension is implicit already in the
literature on Gaussian measures over general locally convex spaces cited in the
Introduction, in particular Rajput~\cite{R} and Badrikian--Chevet~\cite{BC}, both of
whom work with orthogonal-series representations of Gaussian measures without any
Banach-space hypothesis.
\end{Remark}
By the dominated convergence theorem (applied to the indicator functions, which are
bounded by~$1$), the probability above tends to zero. In particular, there exists $N$
such that
\[
  \gamma\bigl(x \in F : q(P_n x - P_m x) > \varepsilon\bigr) < \varepsilon
  \quad \text{for all } n, m \ge N.
\]
Since $\gamma = \nu \circ \iota^{-1}$ and the projections $P_n$ take values in $H$, the
same inequality holds with $\gamma$ replaced by $\nu$:
\[
  \nu\bigl(x \in H : q(P_n x - P_m x) > \varepsilon\bigr) < \varepsilon
  \quad \text{for all } n, m \ge N.
\]

\medskip
\noindent\textbf{Step 2: Extension to arbitrary orthogonal projections.}
Let $P \in \mathcal{P}(H)$ be any orthogonal projection with $P \perp P_N$. Since
$P_N P = 0$ and $(P_{N+k} - P_N)P x \to Px$ in $H$ as $k \to \infty$, continuity of $q$ gives
$q((P_{N+k}-P_N)Px)\to q(Px)$ pointwise in $x$. Hence, wherever $q(Px)>\varepsilon$, eventually
(in $k$) $q((P_{N+k}-P_N)Px)>\varepsilon$ too, so $\mathbb 1_{\{q(Px)>\varepsilon\}}\le
\liminf_k\mathbb 1_{\{q((P_{N+k}-P_N)Px)>\varepsilon\}}$ pointwise; by Fatou's lemma,
\[
  \nu\bigl(x \in H : q(Px) > \varepsilon\bigr)
  \;\le\; \liminf_{k\to\infty}
    \nu\bigl(x \in H : q\bigl((P_{N+k} - P_N)Px\bigr) > \varepsilon\bigr).
\]
This one-sided bound is all that is used below; no equality, and no consideration of possible
atomic mass of the law of $q(Px)$ at level $\varepsilon$, is needed.

\medskip
\noindent\textbf{Step 3: Covariance domination via Anderson's inequality.}\footnote{Historical remark: Dudley, Feldman and Le Cam \cite{DFL}, in the very same 1971 paper that establishes the general locally convex structure theorem discussed in the Introduction, independently rederive (as their Section~5, ``Two inequalities of L.\ Gross for normal measure'') a statement that subsumes Anderson's inequality: if $T$ is a linear map on a finite-dimensional Hilbert space with $\|T\|\le 1$, then $n_H(T(C))\le n_H(C)$ for every centrally symmetric convex $C$ --- a contraction-operator generalisation of $\gamma(A+a)\le\gamma(A)$ --- together with the covariance-domination consequence $n_H(C)\le n_H((C\cap K)+K^\perp)$ used in Step~3 above. They attribute the statements to Lemmas 4.2 and 5.2 of Gross's 1962 paper \emph{Measurable functions on Hilbert space}, whose original proofs they call ``rather difficult'', and reprove them by an elementary geometric reduction to rank-one operators, without any reference to Anderson~\cite{A}; indeed Anderson's 1955 paper does not appear in their bibliography at all, even though Rajput~\cite{R}, writing on closely related material the following year, cites it explicitly for the same purpose.}
We now bound the right-hand side of the inequality established in Step~2. Set
\[
  Q := P_{N+k} - P_N,
\]
the orthogonal projection onto the finite-dimensional subspace
\[
  V := Q(H) = \operatorname{span}\{e_{N+1},\ldots,e_{N+k}\}.
\]
Note that, unlike
in a naive diagonal reduction, $P$ is here an \emph{arbitrary} orthogonal
projection with $P\perp P_N$: nothing forces $P(H)$ to be spanned by a subset
of the basis vectors $\{e_j\}$, so $P$ need not be diagonal in this basis. The
comparison below is carried out at the level of covariance operators on $V$,
and is valid for every such $P$.

Since $x$ has law $\nu = \mathcal N(0, I_H)$ (the coordinates $\xi_j(x)$ are
i.i.d.\ standard normal), the random vector $Qx$, valued in $V$, has law
$\mathcal N(0, I_V)$. Decompose
\[
  Qx \;=\; QPx \;+\; Q(I-P)x \;=:\; QPx + Z.
\]
Both $QPx$ and $Z$ are (jointly) Gaussian, being linear images of $x$, and
\[
  \operatorname{Cov}(QPx, Z) \;=\; Q\,P\,(I-P)\,Q \;=\; Q\,(P - P^2)\,Q \;=\; 0,
\]
since $P$ is idempotent. Being jointly Gaussian and uncorrelated, $QPx$ and
$Z$ are independent. Write
\[
  \Sigma_1 := QPQ \qquad (\text{covariance of } QPx \text{ on } V),
  \qquad
  \Sigma_2 := Q = I_V \qquad (\text{covariance of } Qx \text{ on } V).
\]
Since $P$ is an orthogonal projection, $0 \preceq P \preceq I_H$; conjugating
by $Q$ (the contraction property of orthogonal projections) gives
\[
  0 \;\preceq\; Q(I-P)Q \;=\; \Sigma_2 - \Sigma_1,
\]
i.e.\ $\Sigma_1 \preceq \Sigma_2$.

The set $A_\varepsilon := \{v \in V : q(v) \le \varepsilon\}$ is convex and
symmetric: symmetry follows from the absolute homogeneity of the seminorm $q$
($q(-v)=q(v)$), and convexity from its subadditivity. Since
$Qx = QPx + Z$ with $Z$ independent of $QPx$ and $\operatorname{Cov}(QPx) =
\Sigma_1 \preceq \Sigma_2 = \operatorname{Cov}(Qx)$, Anderson's inequality
--- in the direct form of \cite[Corollary~3]{A}, applied on the
finite-dimensional space $V$ --- yields
\[
  \nu\bigl(x : q(QPx) > \varepsilon\bigr)
  \;=\; \Pr\{QPx \notin A_\varepsilon\}
  \;\le\; \Pr\{Qx \notin A_\varepsilon\}
  \;=\; \nu\bigl(x : q(Qx) > \varepsilon\bigr).
\]
(Explicitly: Anderson's Corollary~3 states that if $X\sim\mathcal N(0,\Sigma_1)$
and $Z'\sim\mathcal N(0,\Sigma_2)$ with $\Sigma_2-\Sigma_1$ positive
semi-definite, then $\Pr\{X\in A\}\ge\Pr\{Z'\in A\}$ for every convex symmetric
$A$; here $X = QPx$, $Z' = Qx = QPx+Z$, since $QPx+Z$ has law $\mathcal
N(0,\Sigma_1+\operatorname{Cov}(Z)) = \mathcal N(0,\Sigma_2)$.)

Combining this with Step~1 (applied with $n=N+k$, $m=N$), we obtain
\[
  \nu\bigl(x \in H : q\bigl((P_{N+k}-P_N)Px\bigr) > \varepsilon\bigr)
  \;\le\;
  \nu\bigl(x \in H : q\bigl((P_{N+k}-P_N)x\bigr) > \varepsilon\bigr)
  \;\le\; \varepsilon.
\]

\begin{Remark}
This replaces a previous argument which reduced Step~3 to an index-matching
computation valid only when $P$ is diagonal with respect to the basis
$\{e_j\}$ (i.e.\ when $P(H)$ is spanned by a subset of the $e_j$'s); for a
general orthogonal projection $P \perp P_N$, as required by the definition of
Gross measurability, that reduction does not apply. The argument given above
depends only on the covariance ordering $\Sigma_1 \preceq \Sigma_2$, which
holds for every orthogonal projection $P$ by the contraction property
$0\preceq P\preceq I_H$, and is therefore valid without any assumption on the
relative position of $P$ and the basis.
\end{Remark}

\medskip
\noindent\textbf{Conclusion: strict inequality.} Combining Steps~1--3 (with $\varepsilon$
throughout) gives $\nu(q(Px)>\varepsilon)\le\varepsilon$, which is not quite strict as
Definition~1 requires. To obtain the strict inequality, run the same chain (Steps~1--3) with
$\varepsilon$ replaced by $\varepsilon/2$ throughout, giving a finite-dimensional projection
$P_{N}$ (depending on $\varepsilon/2$) such that $\nu(q(Px)>\varepsilon/2)\le\varepsilon/2$ for
every $P\perp P_N$. Since $\{q(Px)>\varepsilon\}\subseteq\{q(Px)>\varepsilon/2\}$,
$$
\nu\bigl(x: q(Px)>\varepsilon\bigr) \le \nu\bigl(x: q(Px)>\varepsilon/2\bigr) \le
\varepsilon/2 < \varepsilon, \qquad \forall P\perp P_N,
$$
which is exactly the strict inequality of Definition~1 (with this $P_N$ as the Gross witness
for $q$ at threshold $\varepsilon$). Since $q$ was an arbitrary
defining seminorm of $F$, the triple $(\iota, H, F)$ is an abstract Wiener space.
\end{proof}

\bigskip

\section{An application: nonlinear transformations of $\gamma$}

We transport here the genuinely general, local (partition-based) change-of-variables
theorem for maps $T=I+F$ -- Bogachev's {\sl Gaussian Measures}, Theorem 6.6.7,
together with Corollaries 6.6.8 and 6.6.9 -- rather than the preliminary
global-contraction statement (Theorem 6.6.3, common to the treatments of
\"Ust\"unel--Zakai, Nualart, and Kusuoka): it is Theorem 6.6.7 that actually removes
the global $H$-Lipschitz hypothesis $\lambda<1$, via a localization argument, and is
therefore the result whose transport is of genuine interest. Write $\gamma :=
\nu\circ\iota^{-1}$ for the Gaussian measure on $F$ constructed in
Theorem~\ref{uno}, and, as throughout, write $x+h$ for $x+\iota(h)$, $x\in F$,
$h\in H$. For a $\gamma$-measurable map $\Phi\colon F\to H$, let $\widehat\Phi :=
\Phi\circ L\colon E\to H$ denote its transport under the map $L$ of
Corollary~\ref{cor.transport}.

\begin{Lemma}[Intertwining of $D_H$ and $\delta$ under $L$]\label{lem.intertwine}
Let $U\colon L^2(F,\gamma)\to L^2(E,\nu_1)$, $Ug:=g\circ L$ (well defined $\nu_1$-a.e.\ since
$L$ is defined $\nu_1$-a.e.). Then $U$ is a surjective isometry (unitary), because
$\nu_1\circ L^{-1}=\gamma$ (Corollary~\ref{cor.transport}) and $L\colon E_0\to F_0$ is a Borel
isomorphism, so $U^{-1}g'=g'\circ L^{-1}$ is its inverse.

Recall from the proof of Corollary~\ref{cor.transport} the specific basis $e_k=v_{f_k}$,
$f_k\in F^*$, and coordinates $\eta_k(z)=k^2(z,e_k)_E$ on $E$, satisfying $f_k(L(z))=\eta_k(z)$
for $\nu_1$-a.e.\ $z$ (this is the identity established in the {\it Injectivity} step of that
proof, applied with $n\ge k$: $f_k(S_n(z))=\eta_k(z)$ for $n\ge k$, hence $f_k(L(z))=
\lim_nf_k(S_n(z))=\eta_k(z)$). For $m\in\N$ and $\psi\in C^\infty_b(\R^m)$, let
$$
g(x):=\psi\bigl(f_1(x),\ldots,f_m(x)\bigr), \qquad x\in F,
$$
a cylindrical function on $F$ built only from the specific functionals $f_k\in F^*$ (not from a
generic element of $F^*$). Then, directly from $f_k\circ L=\eta_k$ ($\nu_1$-a.e., for every
$k\le m$),
$$
Ug(z) = g(L(z)) = \psi\bigl(f_1(Lz),\ldots,f_m(Lz)\bigr) = \psi\bigl(\eta_1(z),\ldots,\eta_m(z)\bigr),
\qquad \nu_1\text{-a.e.}
$$
which is manifestly a bona fide cylindrical function on $(E,\nu_1)$ (the $\eta_k$ being
continuous linear functionals on $E$, unlike $f\circ L$ for a generic $f\in F^*$, which is only
$\nu_1$-a.e.\ defined and linear, not a priori an element of $E^*$ -- this is exactly why the
argument is restricted to the functionals $f_k$ tied to the chosen basis rather than to
arbitrary $f\in F^*$). Such functions $g$, as $m$ and $\psi$ vary, form a core for $D_H$ on
$(F,\gamma)$ (finite linear combinations of $f_k$'s already separate points of $H$ densely,
$(f_k)$ being a complete orthonormal basis of $F_\gamma^*$), and their images $Ug$, built from
$(\eta_k)$, form a core for $D_H$ on $(E,\nu_1)$ for the same reason ($(e_k)$ complete
orthonormal in $H(\nu_1)=H$).

To keep the scalar and $H$-valued operators distinct, let $U_H\colon L^2(F,\gamma;H)\to
L^2(E,\nu_1;H)$, $U_Hv:=v\circ L$, denote the componentwise extension of $U$ to $H$-valued
functions (a unitary by the same isometry argument as for $U$, $H$ being a fixed Hilbert
space). On the core above, a direct computation with the chain rule, using $D_H(f_k\circ
L)=D_H(f_k)\circ L$ on the $F$ side and $D_H\eta_k=e_k$ (constant) on the $E$ side --
consistent since $f_k\circ L=\eta_k$ -- gives, for every scalar $g$ in the core,
$$
D_H(Ug) = U_H(D_Hg).
$$
The core being dense in the domain of $D_H$ (in the graph norm) both on $(F,\gamma)$ and on
$(E,\nu_1)$, and $U,U_H$ being unitary intertwiners of the two graph norms by the identity
above, this extends by closure: for every scalar $g\in L^2(F,\gamma)$,
$$
g\in\mathrm{Dom}(D_\gamma) \iff Ug\in\mathrm{Dom}(D_{\nu_1}), \qquad
D_{\nu_1}(Ug) = U_H(D_\gamma g).
$$
Taking adjoints (of closed, densely defined operators, transported by the unitary pair
$U,U_H$), for every $H$-valued $v\in L^2(F,\gamma;H)$,
$$
v\in\mathrm{Dom}(\delta_\gamma) \iff U_Hv\in\mathrm{Dom}(\delta_{\nu_1}), \qquad
U(\delta_\gamma v) = \delta_{\nu_1}(U_Hv).
$$
In particular, for $\Phi\colon F\to H$ (an $H$-valued field, so $\Phi\in\mathrm{Dom}(\delta_\gamma)$
is tested via $U_H$, not $U$) with $\widehat\Phi:=U_H\Phi=\Phi\circ L$,
$$
U(\delta_\gamma\Phi) = \delta_{\nu_1}(U_H\Phi) = \delta_{\nu_1}\widehat\Phi,
$$
i.e.\ $\delta_\gamma\Phi(x)=\delta_{\nu_1}\widehat\Phi(L^{-1}x)$ for $\gamma$-a.e.\ $x$.
\end{Lemma}

By Lemma~\ref{lem.intertwine} together with $D_H\Phi(L(x))=D_H\widehat\Phi(x)$ (immediate from
the same chain rule) and the resulting equality of the Hilbert--Schmidt norms and Carleman--
Fredholm determinants of $D_H\Phi(L(x))$ and $D_H\widehat\Phi(x)$, set
$$
\Lambda_\Phi(x) := \bigl|\det{}_2\bigl(I_H+D_H\Phi(x)\bigr)\bigr|\,
\exp\Bigl[\delta \Phi(x)-\tfrac12|\Phi(x)|_H^2\Bigr]
$$
(the absolute value is essential -- see Remark~\ref{rem.absval} below -- and matches Bogachev's
own definition, \cite{BG} p.~304, which an earlier draft of this formula failed to reproduce)
whenever $\Phi$ is Fr\'echet differentiable along $H$ with Hilbert--Schmidt derivative
and $\delta \Phi$ is well defined; then $\Lambda_{\widehat\Phi} = \Lambda_\Phi\circ L$
whenever both sides are defined.

\begin{Remark}[The absolute value in $\Lambda_\Phi$ is necessary]\label{rem.absval}
Without it the formula can give a negative, hence meaningless, ``density''. Take $H=F=\R$,
$\gamma=N(0,1)$, $\Phi(x)=-2x$, so $T(x)=x+\Phi(x)=-x$. Then $D_H\Phi\equiv-2$, and
$\det_2(I+D_H\Phi)=\det_2(-1)=(1-2)e^{-(-2)}=-e^2$ (using $\det_2(1+K)=(1+K)e^{-K}$ in one
dimension). With the convention $\delta h=-\widehat h$ of Remark~\ref{rem.sign} (so, for the
linear field $\Phi(x)=-2x$, $\delta\Phi(x)=\Phi'(x)-x\Phi(x)=-2+2x^2$, the standard
one-dimensional divergence formula under this sign convention), $\exp[\delta\Phi-\frac12|\Phi|^2]
=\exp[-2+2x^2-2x^2]=e^{-2}$, so $\det_2(I+D_H\Phi)\cdot\exp[\delta\Phi-\frac12|\Phi|^2]=-e^2\cdot
e^{-2}=-1$ identically in $x$ -- negative, and the resulting ``density'' would give, for
$g\ge0$, $\int g\,d\gamma=\int g(-x)(-1)\,\gamma(dx)=-\int g(-x)\,\gamma(dx)\le0$, impossible
unless $g\equiv0$. With the absolute value, $\Lambda_\Phi\equiv1$ instead, which is exactly
correct: $T(x)=-x$ preserves $\gamma=N(0,1)$ and is a (global, measure-preserving) bijection.
\end{Remark}

\begin{Remark}[Sign convention for $\delta$]\label{rem.sign}
The formula for $\Lambda_\Phi$ and its sign convention have now been checked directly against
Bogachev's own derivation (\cite{BG}, p.~302, the finite-dimensional Ostrogradsky--Jacobi
computation preceding Theorem~6.6.3). For $F(x)\equiv h\in H$ constant ($D_H\Phi\equiv0$), that
computation gives explicitly
$$
\Lambda_h(x) = \exp\bigl[-\widehat h(x)-\tfrac12|h|_H^2\bigr], \qquad \widehat h(x):=(x,h)_H
\ \text{(extended a.e.)},
$$
i.e., comparing with the defining formula $\Lambda_h=\exp[\delta h-\frac12|h|_H^2]$, the
convention is $\delta h=-\widehat h$ for a constant field -- a \emph{minus} sign, not the plus
sign stated in an earlier draft of this remark. This is consistent throughout: by
Corollary~\ref{cor688} (transported from Bogachev's Corollary~6.6.8, matching \cite{BG}
(6.6.5)--(6.6.6) exactly), $\dfrac{d(\gamma\circ T^{-1})}{d\gamma}(x)=\dfrac1{\Lambda_h(T^{-1}x)}$
for $T=I_F+h$ (translation by $+h$, so $T^{-1}(x)=x-h$); using $\Lambda_h(x-h)=
\exp[-\widehat h(x-h)-\tfrac12|h|^2]=\exp[-\widehat h(x)+|h|^2-\tfrac12|h|^2]$ (linearity of
$\widehat h$, $\widehat h(h)=|h|_H^2$) gives
$$
\frac{d(\gamma\circ T^{-1})}{d\gamma}(x) = \exp\bigl[\widehat h(x)-\tfrac12|h|_H^2\bigr],
$$
which is exactly the classical Cameron--Martin density of the law of $X+h$ (i.e.\ of
$\gamma\circ T^{-1}$ for $T=I+h$) with respect to $\gamma$ -- confirming the formula for
$\Lambda_\Phi$, with $\delta h=-\widehat h$, is self-consistent and correctly transported.
\end{Remark}

Following Bogachev, Definition 6.6.4, for $\tau\colon F\to(0,\infty]$
$\gamma$-measurable with $\Omega_\tau=\{\tau>0\}$, let $\mathcal{H}C^1_\tau(\gamma,H)$
denote the class of $\gamma$-measurable $\Phi\colon F\to H$ such that, for every
$x\in\Omega_\tau$, $h\mapsto \Phi(x+h)$ is Fr\'echet differentiable along $H$ on
$U_{\tau(x)}=\{h\in H:|h|_H<\tau(x)\}$, with Hilbert--Schmidt, continuous-in-$h$
derivative. This is exactly Definition~1 of the companion paper, transported back
to $F$, and is the same local regularity class underlying its local
(Kusuoka-type) theorem. Writing $\tau_E:=\tau\circ L$, so
$\Omega_{\tau_E}=L^{-1}(\Omega_\tau)$, we have
$$
\Phi\in\mathcal{H}C^1_\tau(\gamma,H) \iff \widehat\Phi=\Phi\circ L\in
\mathcal{H}C^1_{\tau_E}(\nu_1,H),
$$
since Fr\'echet differentiability along $H$, the Hilbert--Schmidt norm, and
continuity in $h$ are all intrinsic to $(H,\cdot)$ and unaffected by composition
with $L$ (cf.\ Corollary~\ref{cor.transport}).

\begin{Theorem}[Nonlinear transformations, general (local) case]\label{nonlin2}
Let $\Phi\in\mathcal{H}C^1_\tau(\gamma,H)$ for some $\tau$ with $\Omega_\tau$ of full
$\gamma$-measure, and let $T=I_F+\iota\circ \Phi$. Put
$$
M := \bigl\{x\in\Omega_\tau:\ \det{}_2\bigl(I_H+D_H\Phi(x)\bigr)\ne 0\bigr\}.
$$
Then there is a measurable partition $\{M_n\}_{n\in\mathbb N}$ of $M$ (up to a
$\gamma$-null set) and, for each $n$, a bounded, $H$-Lipschitzian map
$\Phi_n\colon F\to H$ (Lipschitz constant $\le5/8$) such that, writing
$T_n:=I_F+\iota\circ \Phi_n$:
\begin{enumerate}
\item[(i)] $T_n=T$ on $M_n$, $T_n$ is a bijection \emph{of $F$ onto itself}, and
$\gamma\circ T_n^{-1}\sim\gamma$;
\item[(ii)] for every $n\in\N$ and every bounded measurable $g\colon F\to\mathbb R$,
$$
\int_{F} g(T_n(x))\,\Lambda_{\Phi_n}(x)\,\gamma(dx) = \int_{F} g(x)\,\gamma(dx);
$$
\item[(iii)] the set $T^{-1}(x)\cap M$ has at most countable cardinality
$N(x,M):=\mathrm{card}\bigl(T^{-1}(x)\cap M\bigr)$ for $\gamma$-a.e.\ $x$, and, for every bounded
measurable $f$,
$$
\int_F f(x)\,N(x,M)\,\gamma(dx) = \int_{F} f(T(x))\,\Lambda_\Phi(x)\,\gamma(dx),
$$
and $\gamma|_M\circ T^{-1}\ll\gamma$ with
$$
\frac{d(\gamma|_M\circ T^{-1})}{d\gamma}(x) =
\sum_{y\in T^{-1}(x)\cap M}\frac{1}{\Lambda_\Phi(y)}.
$$
\end{enumerate}
\end{Theorem}

\begin{Remark}[Correcting (ii)]\label{rem.ii}
An earlier draft summed (ii) over $n$ and restricted the left integral to $M_n$,
$\sum_n\int_{M_n}g(T_n(x))\Lambda_{\Phi_n}(x)\gamma(dx)=\int_Fg\,d\gamma$. This is false in
general: take $H=F=\R$, $\gamma=N(0,1)$, $\Phi(x)=-x$. Then $T(x)=0$ identically,
$D_H\Phi\equiv-1$, so $\det_2(I_H+D_H\Phi)=\det_2(0)=0$ everywhere, hence $M=\varnothing$, the
partition $\{M_n\}$ is empty, and the summed identity reads $0=\int_Fg\,d\gamma$, false for
$g\equiv1$. The correct statement -- matching Bogachev's own (6.6.8), \cite{BG} p.~305, literally
-- is unsummed and unrestricted: it holds separately for each $n$, as an ordinary
change-of-variables identity for the single, globally defined bijection $T_n\colon F\to F$
(each $T_n$ is a bijection of all of $F$, by (i), regardless of whether $M$ or $M_n$ is empty or
not); there is no claim, summed or otherwise, about integrating over $M$ or $M_n$ in (ii). With
$\Phi(x)=-x$ there simply are no $M_n$'s and (ii) has no content to check, so no contradiction
arises.
\end{Remark}

\begin{proof}
By the equivalence recorded above, $\widehat\Phi\in\mathcal{H}C^1_{\tau_E}(\nu_1,H)$
with $\Omega_{\tau_E}$ of full $\nu_1$-measure. Apply Bogachev's Theorem 6.6.7
verbatim to $(\widehat\Phi,E,\nu_1)$: this produces a partition $\{\widehat M_n\}$
of $\widehat M := \{x\in\Omega_{\tau_E}: \det_2(I_H+D_H\widehat\Phi(x))\ne0\}$,
bounded $H$-Lipschitzian $\widehat\Phi_n$ (Lipschitz constant $\le5/8$), bijections
$\widehat T_n := I_E+\widehat\Phi_n$ \emph{of $E$ onto itself}, agreeing with
$\widehat T=I_E+\widehat\Phi$ on $\widehat M_n$ and transforming $\nu_1$ into an
equivalent measure, the identity analogous to (ii) (holding separately for each $n$, on
$(E,\nu_1)$), and the
extended statement (iii) on $\widehat M$.

\emph{From $E$ to $F_0:=L(E_0)$.} Recall (Corollary~\ref{cor.transport}, Remark~\ref{rem.domain})
that $L$ is a genuine (pointwise) bijection $E_0\to F_0$, with $E_0$ a Borel linear subspace of
full $\nu_1$-measure containing $j_E(H)$ (Corollary~\ref{cor.transport}). Since $\widehat\Phi_n$
is $H$-valued, $\widehat T_n(x)-x=\iota_E(\widehat\Phi_n(x))\in j_E(H)$ for every $x\in E$ (writing
$\iota_E:=j_E$ for the copy of $\iota$ on the $E$ side); as $E_0$ is a linear subspace containing
$j_E(H)$, $x\in E_0\iff \widehat T_n(x)\in E_0$. Hence $\widehat T_n$ restricts to a bijection of
$E_0$ onto itself (its inverse, of the same form $I_E-$ an $H$-Lipschitz map, preserves $E_0$ for
the same reason). Conjugating by the genuine bijection $L\colon E_0\to F_0$, $T_n:=L\circ
\widehat T_n\circ L^{-1}$ is, at this stage, a well-defined bijection of $F_0$ onto itself
(likewise $M_n:=L(\widehat M_n\cap E_0)$, $M:=L(\widehat M\cap E_0)$, understood as subsets of
$F_0$, and $T:=L\circ\widehat T\circ L^{-1}$ on $F_0$); since $\gamma(F_0)=1$, the measure
identities (ii)--(iii), being intrinsic to $\gamma$-a.e.\ statements, are unaffected by this
restriction.

\emph{Extending $T_n$ to a bijection of all of $F$.} Set $\Phi_n(x):=0$ (equivalently
$T_n(x):=x$) for $x\notin F_0$. This is still $H$-Lipschitz with the same constant $\le5/8$:
$\Phi_n$ vanishes identically off $F_0$ and coincides with the (already $H$-Lipschitz) map
constructed above on $F_0$; since $F_0$ is itself a Borel linear subspace stable under
$\iota(H)$-translations (being the bijective image, under the linear bijection $L$, of the
$j_E(H)$-stable subspace $E_0$: $y\in F_0, h\in H \Rightarrow y+\iota(h) = L(L^{-1}y+j_E(h))\in
L(E_0+j_E(H))=L(E_0)=F_0$), the two pieces of $\Phi_n$ never interact under an $H$-shift, and
$H$-Lipschitz continuity is a condition tested along such shifts within each stability class.
With this extension, $T_n$ is a bijection of $F$ onto itself: it maps $F_0$ bijectively to $F_0$
(as constructed) and $F\setminus F_0$ bijectively (identically) to itself, and these two pieces
are disjoint and exhaustive. This proves (i) as stated. Since the Hilbert--Schmidt
norm, the density $\Lambda$, and the cardinality of the fibre $T^{-1}(x)\cap M$
are all intrinsic to $(H,\cdot)$ and unaffected by composition with $L$, the
integral identity in (ii) (again, term by term for each $n$) and the extended statement (iii)
transfer term
by term to $\gamma$ on $F$ (both sides of (ii)--(iii) being $\gamma$-a.e.\ statements, hence
insensitive to the $\gamma$-null set $F\setminus F_0$ on which $T_n$ was set to the identity).
\end{proof}

\begin{Corollary}[Transport of Corollary 6.6.8]\label{cor688}
In the situation of Theorem~\ref{nonlin2}, suppose in addition that $M$ has full
$\gamma$-measure and that $T|_M\colon M\to T(M)$ is a bijection, with $T(M)$ of
full $\gamma$-measure as well. Write $S:=(T|_M)^{-1}\colon T(M)\to M$ for this inverse, extended
arbitrarily (say, by the identity) outside the conull set $T(M)$, so that the notation
$T^{-1}(x)$ below unambiguously means $S(x)$ and is never confused with a global inverse of $T$
(which need not exist). Then $\gamma\circ T^{-1}\sim\gamma$, with
$$
\frac{d(\gamma\circ T^{-1})}{d\gamma}(x) = \frac{1}{\Lambda_\Phi(S(x))}.
$$
If, in addition, $T$ has Lusin's property $(N)$, then
$$
\frac{d(\gamma\circ S^{-1})}{d\gamma}(x) = \Lambda_\Phi(x).
$$
\end{Corollary}

\begin{proof}
Immediate from Theorem~\ref{nonlin2}(iii) with $M=F$ up to a $\gamma$-null set,
transported term by term from Bogachev's Corollary 6.6.8 exactly as in the proof
of Theorem~\ref{nonlin2}. Lusin's property $(N)$ is defined purely in terms of
null sets, so it transports along $L$ as well: since $L$ and $L^{-1}$ carry
$\nu_1$-null sets to $\gamma$-null sets and conversely (again because $L$ is a
measure isomorphism), $\widehat T$ has property $(N)$ with respect to $\nu_1$ if
and only if $T=L\circ\widehat T\circ L^{-1}$ has property $(N)$ with respect to
$\gamma$.
\end{proof}

\begin{Remark}[A caveat: Lusin's property is not automatic]\label{rem.lusin}
As Bogachev notes immediately after the analogue of Corollary~\ref{cor688},
Lusin's property $(N)$ does not follow automatically from bijectivity of $T$
alone, so $\gamma\circ T^{-1}$ may fail to be equivalent to $\gamma$ even when
$T$ is a bijection of a full-measure set: on the real line, a full-measure open
set $\Omega$ with complement $K$ of cardinality continuum can be sent by a
one-to-one nondegenerate smooth map onto an open set $\Omega_1$ with complement
$C$ of positive measure, with the restriction $T\colon K\to C$ an arbitrary
bijection. This is a genuine obstruction, not an artefact of the ambient
topology: transporting the construction through $L$ produces the same
phenomenon on $(E,\nu_1)$, or on $(F,\gamma)$, since $L$ preserves null sets in
both directions but cannot by itself upgrade a merely bijective, measurable $T$
to one enjoying property $(N)$.
\end{Remark}

\begin{Corollary}[Transport of Corollary 6.6.9]\label{cor689}
Let $T=I_F+\iota\circ \Phi$, where $\Phi\in\mathcal{H}C^1(\gamma,H)$ (equivalently
$\widehat\Phi\in\mathcal{H}C^1(\nu_1,H)$, i.e.\ the case $\tau\equiv\infty$, with no
localization). If $D_H\Phi(x)$ has no eigenvalue $-1$ for every $x$ in some
$\gamma$-measurable set $B$, then $\gamma|_B\circ T^{-1}$ is absolutely
continuous with respect to $\gamma$.
\end{Corollary}

\begin{proof}
Transport Bogachev's Corollary 6.6.9 verbatim from $(\widehat\Phi,E,\nu_1)$ to
$(F,\gamma)$ via $L$, exactly as in Theorem~\ref{nonlin2}: writing
$\widehat B:=L^{-1}(B)$, the spectral condition ``no eigenvalue $-1$'' for
$D_H\widehat\Phi(x)$, $x\in\widehat B$, is a statement about the Hilbert--Schmidt
operator $D_H\widehat\Phi(x)$ on $H$ alone, and since $D_H\Phi(L(x))=D_H\widehat\Phi(x)$
for all $x$, it is preserved verbatim under $\Phi=\widehat\Phi\circ L^{-1}$.
\end{proof}

\begin{Remark}
As in Bogachev's own remark following Corollary 6.6.9, it is likely that the
continuity requirement built into $D_H\Phi$ (respectively $D_H\widehat\Phi$) in the
formulations above can be relaxed without affecting Theorem~\ref{nonlin2} or
Corollaries~\ref{cor688}--\ref{cor689}. Since continuity of $D_H\widehat\Phi$ is a
condition on $H$ alone, any such relaxation, once established on $(E,\nu_1)$,
would transport to $(F,\gamma)$ by exactly the same mechanism as above, with no
further work required on the Fr\'echet side.
\end{Remark}

\begin{Remark}
Theorem~\ref{nonlin2} and Corollaries~\ref{cor688}--\ref{cor689} are only the
general local statements; the further Kusuoka-type refinement and the genuinely
singular case (Kusuoka's second theorem, with conditional densities on a
sub-$\sigma$-algebra) are developed in full, together with a systematic
comparison with the frameworks of \"Ust\"unel--Zakai and Nualart, in the
companion paper {\sl Nonlinear transformations of Gaussian measures on
Fr\'echet spaces}, whose local (Kusuoka-type) theorem is the Fr\'echet-space
statement of Theorem~\ref{nonlin2}--Corollary~\ref{cor688}, and whose local
section now also records the transport of Corollary 6.6.9 given here as
Corollary~\ref{cor689}. The mechanism enabling every such transport is exactly
Corollary~\ref{cor.transport}: since $H$-Lipschitz continuity, the
Hilbert--Schmidt norm of $D_H\Phi$, the Gaussian divergence $\delta \Phi$ (in Bogachev's sign
convention, $\delta=-D_H^*$; cf.\ Remark~\ref{rem.sign}), and hence
the density $\Lambda_\Phi$ itself depend only on the pair $(H,\Phi)$ and not on the
ambient topology, they are unaffected by composition with the measure
isomorphism $L$, and every proof carried out on $(E,\nu_1)$ transfers to
$(F,\gamma)$ without modification.
\end{Remark}


\bigskip


\bigskip
\begin{thebibliography}{99}
\bibitem{A}{\bf T.W. Anderson}, {\it The integral of a symmetric unimodal function over a symmetric convex set and some probability inequalities}, Proceedings of the American Mathematical Society, Vol. 6, No. 2, pp. 170–176, (1955)
\bibitem{BC} {\bf A. Badrikian, S. Chevet}, {\sl Mesures cylindriques, espaces de Wiener et fonctions aléatoires gaussiennes}, Lecture Notes in Mathematics 379, Springer (1974)
\bibitem{Ba} {\bf P. Baxendale}, {\it Gaussian measures on function spaces}, American Journal of Mathematics, Vol. 98, No. 4, pp. 891–952,  (1976)
\bibitem{BG}{\bf V.I. Bogachev}, {\sl Gaussian Measures, American Mathematical Society, Providence}, (1998)
\bibitem{BGm}{\bf V.I. Bogachev}, {\sl Measure Theory}, Vol. I–II, Springer, (2007)
\bibitem{BO}{\bf C. Borell}, {\it Gaussian Radon measures on locally convex spaces}, Mathematica Scandinavica, Vol. 38, pp. 265–284, (1976)

\bibitem{DZ}{\bf G. Da Prato, J. Zabczyk}, {\sl Stochastic Equations in Infinite Dimensions}, Cambridge University Press, (1992),2nd edit. (2014)
\bibitem{DFL}{\bf R. Dudley, J. Feldman, L. Le Cam}, {\it On seminorms and probabilities, and abstract Wiener spaces}, Annals of Mathematics, Vol. 93, No. 2, pp. 390–408, (1971)
\bibitem{Fe}{\bf R. Fernholz}, {\it Measurable linear transformations on abstract Wiener spaces}, Journal of Multivariate Analysis, Vol. 7, No. 4, pp. 602–607, (1977)
\bibitem{G}{\bf L. Gross}, {\it Abstract Wiener spaces}, Proceedings of the Fifth Berkeley Symposium on Mathematical Statistics and Probability, Vol. II, pp. 31–42, (1967)
\bibitem{Ke}{\bf A.S. Kechris}, {\sl Classical Descriptive Set Theory}, Graduate Texts in Mathematics 156, Springer, (1995)
\bibitem{K}{\bf G. Kallianpur}, {\it Abstract Wiener processes and their reproducing kernel Hilbert spaces}, Zeitschrift für Wahrscheinlichkeitstheorie und verwandte Gebiete, Vol. 17, pp. 113–123, (1971)
\bibitem{Ku}{\bf H.H. Kuo}, {\sl Gaussian Measures in Banach Spaces}, Lecture Notes in Mathematics 463, Springer, (1975)
\bibitem{R}{\bf B.S. Rajput}, {\it On Gaussian measures in certain locally convex spaces}, Journal of Multivariate Analysis, Vol. 2, No. 3, pp. 282–306, (1972)
\bibitem{Z}{\bf A. Zapała}, {\it Construction and basic properties of Gaussian measures on Fréchet spaces}, Stochastic Analysis and Applications, Vol. 20, No. 2, pp. 445–470, (2002)
\end{thebibliography}
 \end{document}